\documentclass[11pt]{amsart}
\usepackage{graphicx}
\usepackage{amsmath,amsfonts,amssymb,latexsym,amsthm,amscd}
 \newtheorem{theorem}{Theorem}[section]
 \newtheorem{lemma}[theorem]{Lemma}

 \newtheorem{remark}[theorem]{Remark}

\begin{document}

\title{Generating the mapping class groups by torsions}

\author{Xiaoming Du}
\address{South China University of Technology,
  Guangzhou 510640, P.R.China}
\email{scxmdu@scut.edu.cn}


\thanks{This research is supported by NSFC (Grant No. 11401219 and 11471248) and by the
Fundamental Research Funds for the Central Universities in China.}

\maketitle

\begin{abstract}
  Let $S_g$ be a closed oriented surface of genus $g$ and let
  $\text{Mod}(S_g)$ be the mapping class group. When the genus is
  at least 3, $\text{Mod}(S_g)$ can be generated by torsion elements.
  We prove the following results: For $g \geq 4$, $\text{Mod}(S_g)$ can
  be generated by 4 torsion elements. Three generators are involutions and the
  forth one is an order 3 element. $\text{Mod}(S_3)$ can be generated
  by 5 torsion elements. Four generators are involutions and the
  fifth one is an order 3 element.
\end{abstract}

\section{Introduction}
Let $S_g$ be a closed oriented surface of genus $g$. The mapping class
group $\text{Mod}(S_g)$ is defined by $\text{Homeo}^{+}(S_g)/\text{Homeo}_{0}(S_g)$,
the group of homotopy classes of oriented-preserving homeomorphisms of
$S_g$.

The study of the generating set of the mapping class group of a closed
oriented surface begun in the 1930s. The first generating set of the mapping
class group was found by Dehn (\cite{De}). This generating set consist of
$2g(g-1)$ Dehn twists for genus $g \geq 3$. About a quarter of a century
later, Lickorish found a generating set consisting of $3g-1$ Dehn twists
for $g \geq 1$ (\cite{Li}). The number of Dehn twists in the generating
set was improved to $2g+1$ by Humphries (\cite{Hu}). In fact this is the
minimal number of the generators if we require that all of the generators
are Dehn twists. This is also proved by Humphries.

If we consider generators other than Dehn twists, it is possible to find
smaller generating sets. Wajnryb found that the mapping class group
$\text{Mod}(S_g)$ can be generated by two elements (\cite{Wa}), each of
which is some product of Dehn twists. One element in Wajynrb's generating
set has finite order. Korkmaz in \cite{Ko} gave another generating set
consisting of two elements. One of Korkmaz's generators is a Dehn twist
and the other one is a torsion element of order $4g+2$.

The mapping class group $\text{Mod}(S_g)$ can be generated only by torsion
elements. By an $involution$ in a group we simply mean any element of
order 2. Luo proved that for $g \geq 3$, $\text{Mod}(S_g)$ is generated
by finitely many involutions (\cite{Lu}). In Luo's result, the number of
the involutions grows linearly with $g$. Brendle and Farb found a universal
upper bound for the number of involutions that generate $\text{Mod}(S_g)$.
They proved that for $g \geq 3$, $\text{Mod}(S_g)$ can be generated by 7
involutions (\cite{BF}). Kassabov improved the number of involution
generators to 4 if $g \geq 7$, 5 if $g \geq 5$ and 6 if $g \geq 3$
(\cite{Ka}).

For the extended mapping class group, Stukow proved that $\text{Mod}^{\pm}(S_g)$
can be generated by three orientation reversing elements of order 2 (\cite{St}).

If we consider torsion elements of higher orders, not only involutions,
then we can reduce the number of generators. In fact,
Brendle and Farb in \cite{BF} also found a generating set consisting of 3
torsion elements. The minimal generating set of torsion elements was found
by Korkmaz. He showed that $\text{Mod}(S_g)$ can be generated by two torsion
elements, each of which is of order $4g+2$ (\cite{Ko}). For generating set
of smaller order, Monden found that for $g \geq 3$, $\text{Mod}(S_g)$ is
generated by 3 elements of order 3, or generated by 4 elements of order 4
(\cite{Mo}).

Since the group generated by two involution is a dihedral group, Farb and
Margalit in \cite{FM}, Kassabov in \cite{Ka} and Monden in \cite{Mo} mentioned
that the following problem is open:

\vspace{0.2cm}

\noindent \textbf{Problem:}
Whether or not $\text{Mod}(S_g)$ can be generated by three involutions?

\vspace{0.2cm}

In this paper, we give generating sets consisting of 3 involutions and a torsion
of order 3:

\begin{theorem}
\label{Main}

For $g \geq 4$, $\text{Mod}(S_g)$ can be generated by 4 torsion elements, among which
3 elements are involutions and the forth one is an order 3 torsion. For $g=3$,
$\text{Mod}(S_3)$ can be generated by 5 torsion elements, among which 4 elements are
involutions and the fifth one is an order 3 torsion.
\end{theorem}

\begin{remark}

When $g<7$, the number of involutions in Kassabov's result is at least 5.
We give a smaller generating set in these cases. For $g \geq 4$ cases,
our result uses the generating sets of the same form.

\end{remark}

\begin{remark}

Monden's result use only 3 torsions of order 3. Korkmaz's result use only 2
torsions of order $4g+2$. Our advantage is the use of involutions.

\end{remark}

\section{Preliminaries}
We use the convention of functional notation, namely, elements of the mapping class
group are applied right to left. The composition $fh$ means that $h$ is applied first.

Let $c$ be the isotopy class of a simple closed curve on $S_g$. Then the
(left-hand) $Dehn\;twist$ $T_c$ about $c$ is the homotopy class of the homeomorphism
obtained by cutting $S_g$ along $c$, twisting one of the sides by $2\pi$ to the
left and gluing two sides of $c$ back to each other. See Fig. 1.

\begin{figure}[th]
\centerline{\includegraphics[totalheight=2cm]{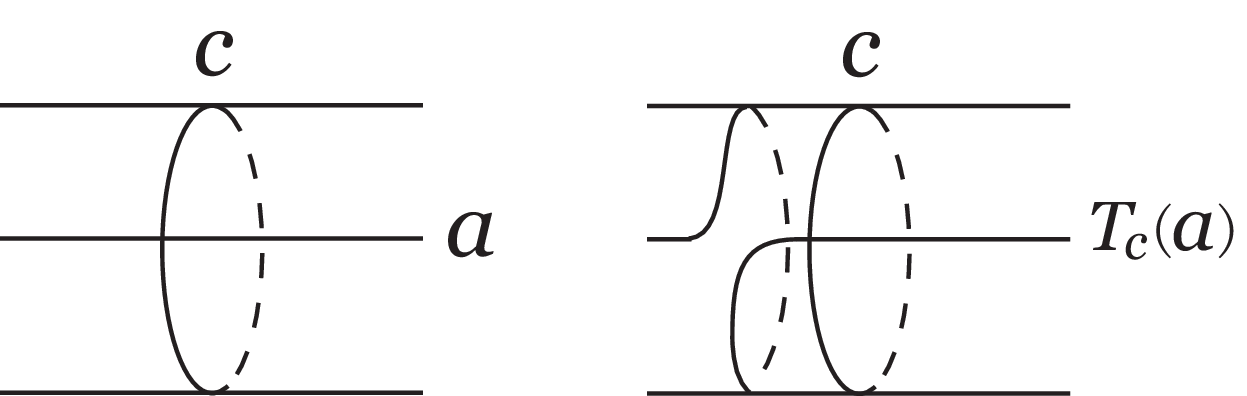}}
\vspace*{8pt}
\caption{Dehn twist.\label{Fig. 1}}
\end{figure}


We recall the following results (see, for instance, \cite{FM}):

\begin{lemma}[Conjugacy relation]

For any $f \in \text{Mod}(S_g)$ and any isotopy class $c$ of simple closed curves
in $S_g$, we have:

$$T_{f(c)}=f\,T_c\,f^{-1}.$$

\end{lemma}

\begin{lemma}[Dehn twists along disjoint curves commute]

Let $a, b$ be two simple closed curves on $S_g$. If $a$ is disjoint from $b$, then

$$T_aT_b=T_bT_a.$$

\end{lemma}

\begin{lemma}[Braid relation]

Let $a, b$ be two simple closed curves on $S_g$. If the geometric intersection
number of $a$ and $b$ is one, then

$$T_aT_bT_a=T_bT_aT_b.$$

\end{lemma}

Let $c_1, c_2, \dots, c_t$ be simple closed curves on $S_g$. If $i(c_j, c_{j+1})=1$
for all $1 \leq j \leq t-1$ and $i(c_j, c_k)=0$ for $|j-k|>1$, then we call
$c_1, c_2, \dots, c_t$ is a $chain$. Take a regular neighbourhood of their union.
The boundary of this neighbourhood consist of one or two simple closed curves,
depending on whether $t$ is even or odd. In the even case, denote the boundary curve
by $d$. In the odd case, denote the boundary curves by $d_1$ and $d_2$.
We have the following $chain$\;$relation$:

\begin{lemma}[Chain relation]
$$
\begin{array}{llll}
  (T_{c_1}T_{c_2}\,\dots\,T_{c_t})^{2t+2} &=& T_d            & t\;even,  \\
  (T_{c_1}T_{c_2}\,\dots\,T_{c_t})^{t+1}  &=& T_{d_1}T_{d_2} & t\;odd.
\end{array}
$$
\end{lemma}

Dehn discovered the lantern relation and Johnson rediscovered it (\cite{Jo}).
This is a very useful relation in the theory of maping class groups.

\begin{lemma}[Lantern relation]

Let $a, b, c, d, x, y, z$ be the curves showed in Fig. 2 on a genus zero
surface with four boundaries. Then

$$T_aT_bT_cT_d=T_xT_yT_z$$

\end{lemma}

\begin{figure}[th]
\centerline{\includegraphics[totalheight=3cm]{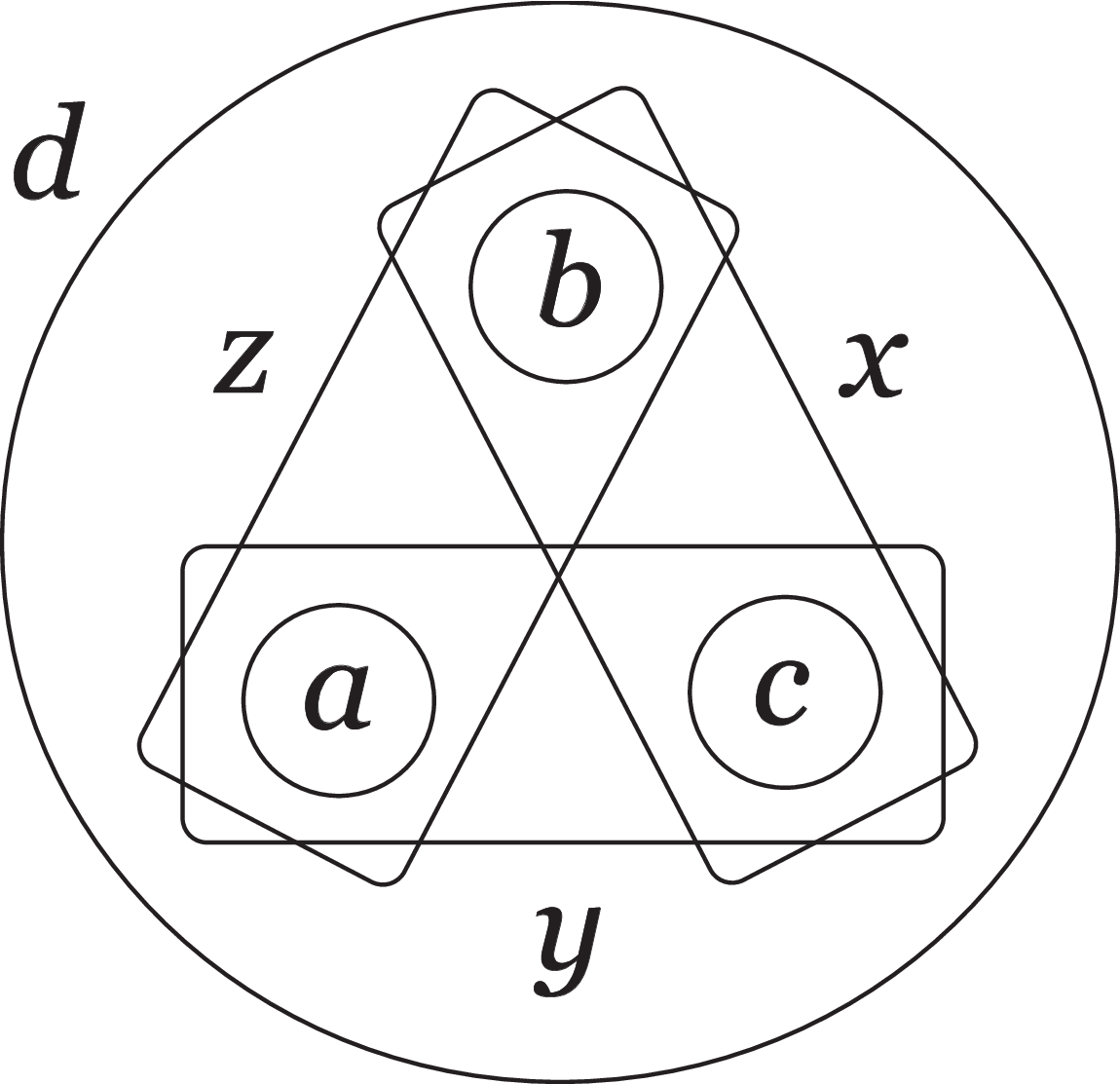}}
\vspace*{8pt}
\caption{lantern relation.\label{Fig. 2}}
\end{figure}

Lickorish's result on the set of generators is the following:

\begin{lemma}[Lickorish's generators]

The mapping class group $\text{Mod}(S_g)$ is generated by the set of Dehn twists
$\{T_{a_i}\,(1 \leq i \leq g)$, $T_{b_i}\,(1 \leq i \leq g)$,
$T_{a_i}\,(1 \leq i \leq g-1)\}$ (See Fig. 3).

\end{lemma}

\begin{figure}[th]
\centerline{\includegraphics[totalheight=5cm]{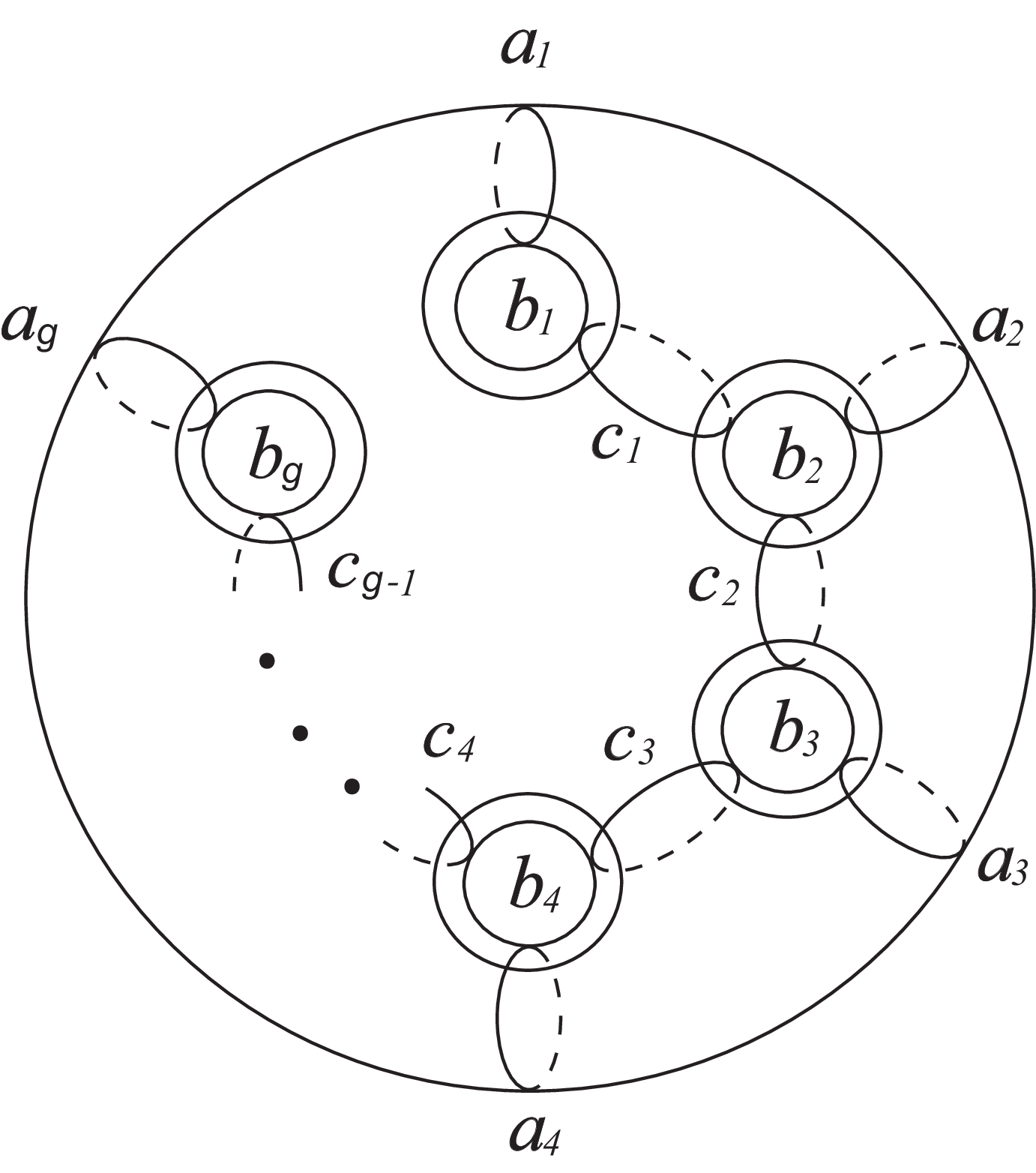}}
\vspace*{8pt}
\caption{Lickorish's curves.\label{Fig. 3}}
\end{figure}

\section{Proof of the main result}

Now we are ready to prove the main result of this paper.

\vspace{0.2cm}

\begin{proof}(The proof of Theorem \ref{Main})

\vspace{0.2cm}

According to Lickorish's theorem, we only need to construct a set of torsion elements
such that each element in the set of Lickorish$'$s generators $\{T_{a_i}\,(1 \leq i \leq g)$,
$T_{b_i}\,(1 \leq i \leq g)$, $T_{c_i}\,(1 \leq i \leq g-1)\}$ can be generated by
these torsion generators. By the conjugacy relation of Dehn twists, if the group $G$ generated
by some torsions has the following two properties, then it is $\text{Mod}(S_g)$:

property\;1, under the action of $G$, ${a_i}\,'s, {b_i}\,'s, {c_i}\,'s$ are in the same orbit;

property\;2, some $T_{a_i}$ is in $G$.

\vspace{0.2cm}

There are involutions $f_1, f_2$ on the surface such that $f_1$ and $f_2$ are
$\pi$-rotations and $f_2f_1$ is an order $g$ element, hence $f_2f_1$ permutes
${a_i}\,'s$, ${b_i}\,'s$, ${c_i}\,'s$ respectively (See Fig. 4).

\begin{figure}[th]
\centerline{\includegraphics[totalheight=7cm]{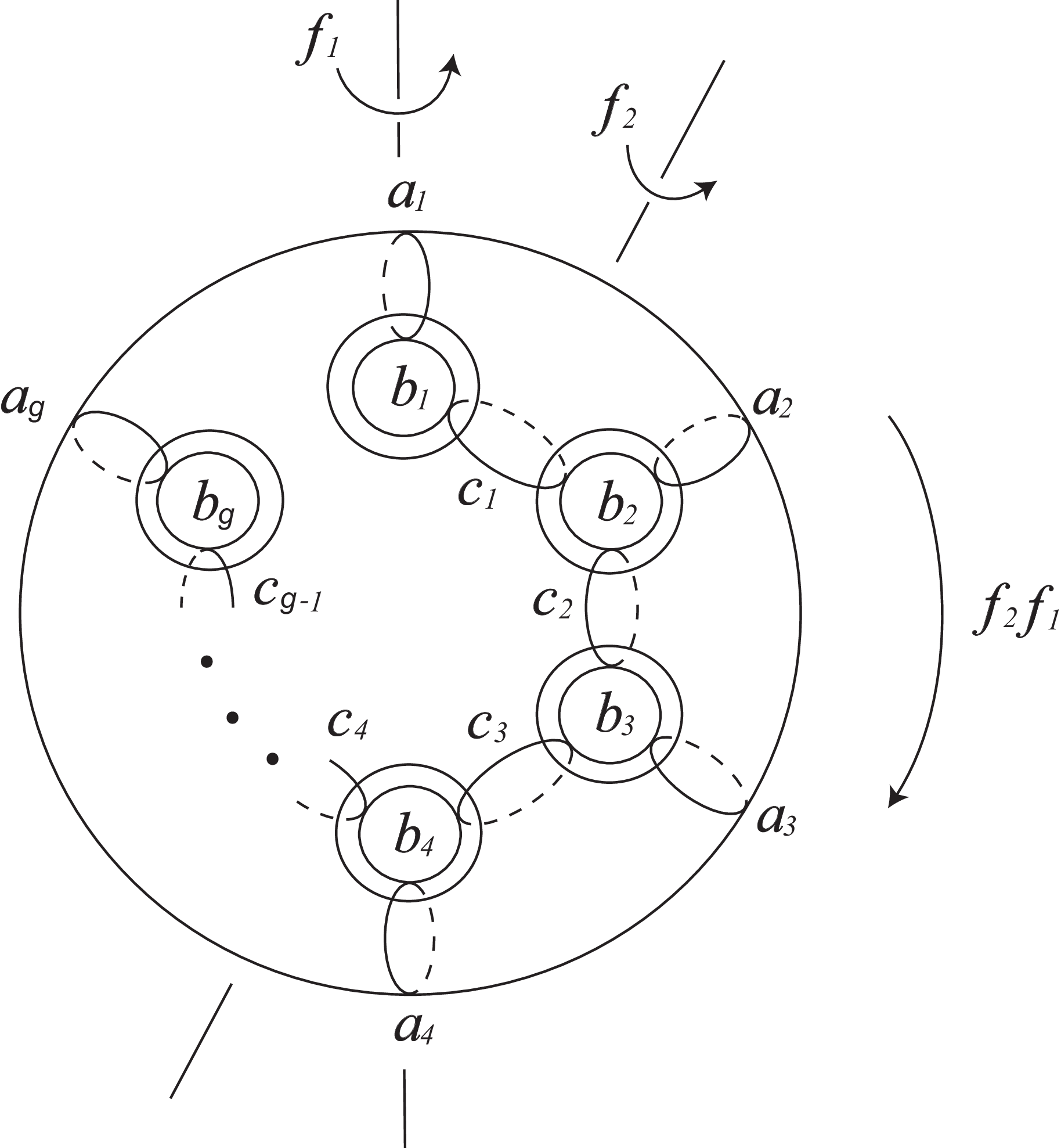}}
\vspace*{8pt}
\caption{involutions.\label{Fig. 4}}
\end{figure}

Meanwhile, in the lantern relation (see the previous Fig. 2), since each curve in
$\{a, b, c, d\}$ is disjoint from $x, y, z$, and each curve in $\{a, b, c, d\}$ is
disjoint from each other, the Dehn twists along them are commutative. The lantern relation
can be rewritten as

$$T_d=(T_xT_a^{-1})(T_yT_b^{-1})(T_zT_c^{-1}).$$

Locally, there is a $(2\pi/3)$-rotation $h$ on the genus zero surface with four boundary,
mapping the curve $d$ to itself, sending $a$ to $b$, $b$ to $c$, $c$ to $a$,
$x$ to $y$, $y $ to $z$ and $z$ to $x$. Hence

$$T_d=(T_xT_a^{-1})(hT_xT_a^{-1}h^{-1})(h^2T_xT_a^{-1}h^{-2}).$$

Compare Fig. 4 to Fig. 2, there is a a genus zero subsurface with four boundaries
$\{a_1$, $c_1$, $c_2$, $a_3\}$. $a_2$ separates $\{a_1, c_1\}$ from
$\{c_2, a_3\}$ on this subsurface. The position of $\{a_1, c_1, a_2\}$ on the
subsurface in Fig. 4 is like the position of $\{a, d, x\}$ in Fig. 2.
$T_d$ and $T_xT_a^{-1}$ in Fig. 2 is like $T_{c_1}$ and $T_{a_2}T_{a_1}^{-1}$.
If we can extend $h$ to a global order 3 homeomorphism $f_3$ on $S_g$, then

$$T_{c_1}=(T_{a_2}T_{a_1}^{-1})(f_3T_{a_2}T_{a_1}^{-1}f_3^{-1})(f_3^2T_{a_2}T_{a_1}^{-1}f_3^{-2}).$$

Now we decompose $T_{a_2}T_{a_1}^{-1}$ into the product of two
involutions $f_2$ and $T_{a_1}f_2T_{a_1}^{-1}$:

$$T_{a_2}T_{a_1}^{-1}=(f_2T_{a_1}f_2)T_{a_1}^{-1}=f_2(T_{a_1}f_2T_{a_1}^{-1}).$$

Notice that $f_2f_1$ send $c_1$ to $c_2$ and $f_3$ permutes $c_2, a_1$ and $a_3$.
Hence the subgroup $G \leq \text{Mod}(S_g)$ generated by $\{f_1$, $f_2$,
$T_{a_1}f_2T_{a_1}^{-1}$, $f_3\}$ includes all the $T_{a_i}\,'s$ and $T_{c_i}\,'s$.
It remains to show how we get $T_{b_i}\,'s$ by torsion elements.
We only need to construct a torsion sending some $a_i$ to $b_i$.

In the case of genus $g \geq 4$, we make the global order 3 homeomorphism $f_3$
sending some $a_i$ to $b_i$ as follows. See Fig. 5.

\begin{figure}[th]
\centerline{\includegraphics[totalheight=4cm]{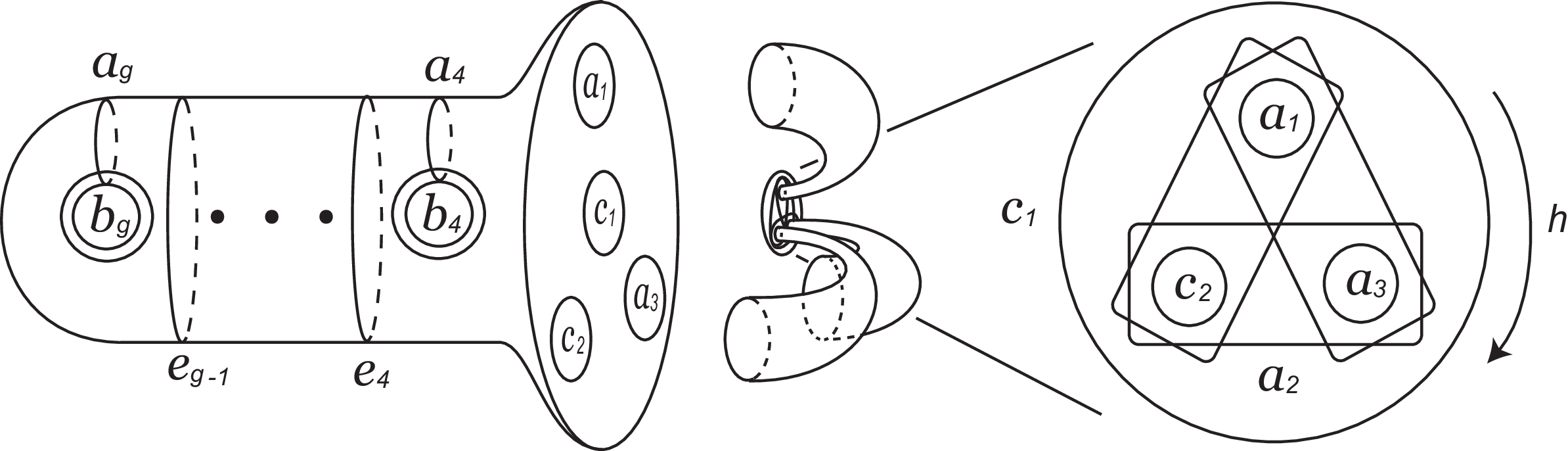}}
\vspace*{8pt}
\caption{order 3 map of the surface.\label{Fig. 5}}
\end{figure}

$h$ is an order 3 local homeomorphism on the genus zero subsurface with four boundaries.
Each point on this subsurface (including the point on the boundary) has period 3
except the fix point at the center.
The complement of the such a subsurface is a genus $g-3$ subsurface with also four
boundaries $a_1, a_3, c_1, c_2$. On the complement, we need to construct a periodic map with the
following properties:
(1) $c_1$ is mapped to itself;
(2) $a_1, a_3, c_2$ are permuted cyclicly;
(3) each point on the complement subsurface (including the point on the boundary)
has period 3 except the fix points;
(4) some $a_i$ is sent to $b_i$.

We cut the complement along curves $e_4, \dots, e_{g-1}$ into $g-3$ pieces, each of which is of genus 1. See Fig. 5.
When $g \geq 5$, such pieces can be divided into 3 classes:
class (i): has only one boundary $e_{g-1}$;
class (ii): has 5 boundaries $e_4, a_1, a_3, c_2, c_1$;
class (iii): has 2 boundaries $e_{i-1}, e_{i}$.
When $g = 4$, the complement has only one piece with 4 boundaries $a_1, a_3, c_2, c_1$.

To construct the periodic maps of order 3 on such pieces, sending some meridian $a_i$ to the
longitude $b_i$, take the genus 1 surfaces as the quotient space of the hexagon (perhaps with holes)
gluing the opposite sides. See Fig. 6. The $2\pi/3$-rotation of the plane is the obvious homeomorphism
on the quotient space we want.

\begin{figure}[th]
\centerline{\includegraphics[totalheight=4.5cm]{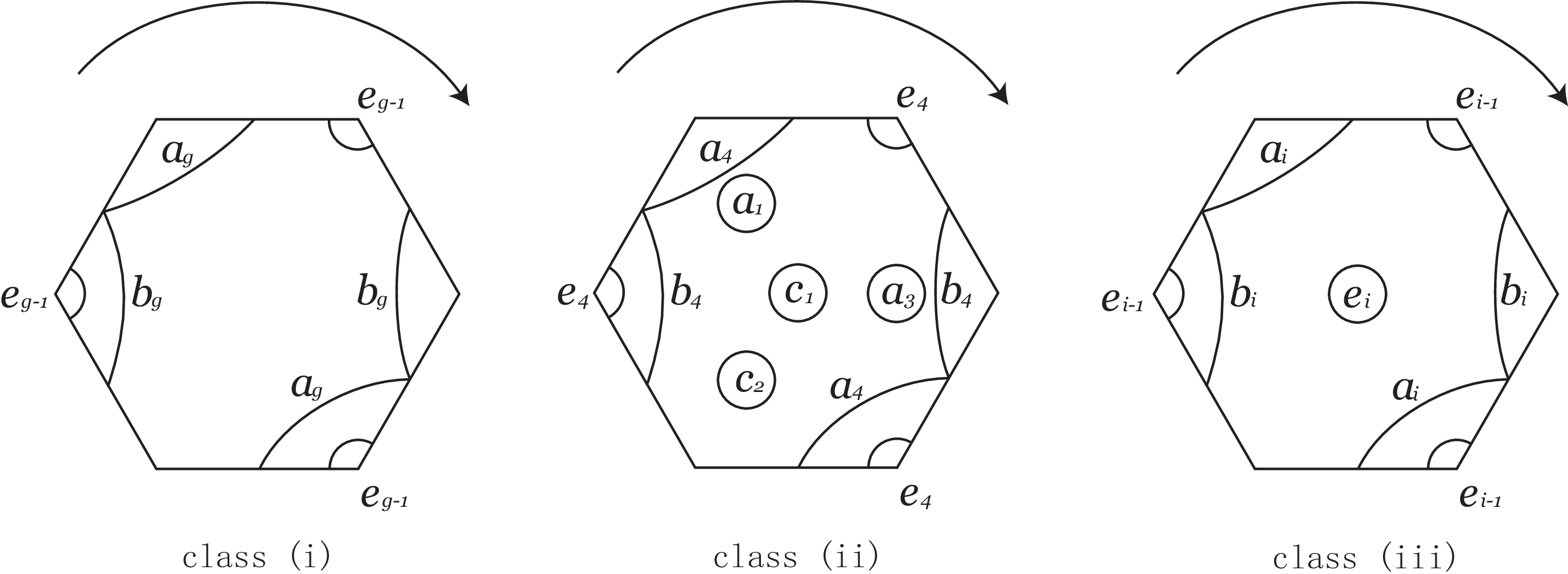}}
\vspace*{8pt}
\caption{order 3 map of the torus.\label{Fig. 6}}
\end{figure}

Now at the beginning, we have the order 3 homeomorphism $h$ on the lantern.
Then we can extend the order 3 homeomorphism to the adjacent piece,
inducing an order 3 homeomorphism on the rest component of the boundary of the adjacent piece.
Then piece by piece, we make a global order 3 homeomorphism $f_3$ of the genus $g$ surface.
All the $a_i\,'s$, $b_i\,'s$, $c_i\,'s$ is in the same orbit under the action of the group
$G$ generated by $\{f_1$, $f_2$, $T_{a_1}f_2T_{a_1}^{-1}$, $f_3\}$.
$T_{a_1}$ is in $G$. Hence $G=\text{Mod}(S_g)$.

In the case of genus $g=3$, the complement of the lantern bounded by $a_1, a_3, c_2, c_1$
is a surface of genus zero having four boundaries. The global order 3 homeomorphism $f_3$ on
complement of the lantern can only be the same form as $h$.
$f_3$ cannot send some $a_i$ to some $b_j$.
We need one more involution to send some $a_i$ to some $b_j$.
Now $a_3$ and $b_3$ are two non-separating curves on the surface.
So there is a homeomorphism $\sigma=T_{a_3}T_{b_3}$ sending $a_3$ to $b_3$ and fixing $a_1, b_1, a_2, b_2$.
Then $\sigma^{-1}f_1\sigma$ is an involution,
sending $a_3$ to $b_2$. The group generated by $\{f_1$, $f_2$, $T_{a_1}f_2T_{a_1}^{-1}$,
$f_3$, $\sigma^{-1}f_1\sigma\}$ is $\text{Mod}(S_3)$.

\end{proof}

\begin{remark}
The idea of the construction of the order 3 mapping class in the proof is based on the method in \cite{Mo} by Naoyuki Monden.
\end{remark}

\end{document}